\documentclass[12pt, reqno]{amsart}
\usepackage{amsmath, amsthm, amscd, amsfonts, amssymb, graphicx, color}
\usepackage[bookmarksnumbered,plainpages]{hyperref}
\newtheorem{theorem}{Theorem}

\newtheorem{proposition}[theorem]{Proposition}
\newtheorem{corollary}[theorem]{Corollary}
\theoremstyle{definition}

\theoremstyle{remark}
\newtheorem{remark}[theorem]{Remark}

\newcommand{\bea}{\begin{eqnarray*}}
\newcommand{\eea}{\end{eqnarray*}}
\begin{document}
\title {The higher duals  of certain class of  Banach algebras}
\author {A.R. KHODDAMI and H.R. EBRAHIMI  VISHKI}
\address{Department of Pure Mathematics, Ferdowsi University of Mashhad, P.O. Box 1159, Mashhad 91775, Iran} \email{khoddami.alireza@yahoo.com}
\address{Department of Pure Mathematics and  Center of Excellence
in Analysis on Algebraic Structures (CEAAS), Ferdowsi University of
Mashhad, P.O. Box 1159, Mashhad 91775, Iran.}
\email{vishki@um.ac.ir}
\begin{abstract}
Given a Banach space $A$ and fix a non-zero $\varphi\in A^*$ with $\|\varphi\|\leq 1$. Then the product $a\cdot b=\langle\varphi, a\rangle\ b$ turning $A$ into a Banach algebra which will be denoted by $_\varphi A.$
Some of the main properties of $_\varphi A$ such as Arens regularity, $n$-weak amenability and semi-simplicity   are investigated.
\end{abstract}
\maketitle
\section{introduction}
This paper had its genesis in a simple beautiful example of Zhang {\cite[Page 507]{Z}}. For an infinite set $S$  he equipped $\l^1(S)$ with the algebra product $a\cdot b=a(s_0) b \ \ (a, b\in\l^1(S))$, where $s_0$ is a fixed element of $S$. He used this as a Banach algebra which is $(2n-1)-$weakly amenable but is not $(2n)-$weakly amenable for any $n\in\mathbb{N}.$  Here we study a more general form of this example. Indeed, we equip  a non-trivial product on a
general  Banach space turning it to a Banach algebra whose  study  has significance in
many respects. For instance,  it can serve as a source of
 (counter-)examples for various purposes in functional
analysis.

Given a Banach space $A$ and fix a non-zero $\varphi\in A^*$ with $\|\varphi\|\leq 1$. Then the product $a\cdot b=\langle\varphi, a\rangle b$ turning $A$ into a Banach algebra which will be denoted by $_\varphi A.$ Trivially $_\varphi A$ has a left identity (indeed, every $e\in\ _\varphi A$ with $\langle\varphi, e\rangle=1$ is a left identity), while it has no bounded approximate identity in the case where $\dim(A)\geq 2.$ Now the Zhang's example can be interpreted as an special case of ours. Indeed, he studied  $_{\varphi_{s_0}}\l^1(S),$ where ${\varphi_{s_0}}\in\l^\infty(S)$ is the characteristic at $s_0$. Here, among other things, we focus on the higher duals of $_\varphi A$ and investigate various notions of amenability for $_\varphi A$. In particular, we prove that for every $n\in\mathbb{N},$ $_\varphi A$ is  $(2n-1)-$weakly amenable but it is not $(2n)-$weakly amenable for any $n$, in the case where $\dim(\ker\varphi)\geq 2$.
\section{ The results}
 Before we proceed for the main results we  need a bit preliminaries. As we shall be concerned  to the Arens products $\square$ and  $\lozenge$ on the bidual $A^{**}$ of a Banach algebra $A$, let us to  introduce these products in the following steps.

 Let $a,b\in A, f\in A^*$  and    $m,n\in  A^{**}.$
$$\begin{array}{ll}
\langle f\cdot a,b \rangle= \langle f,ab \rangle       & \ \ \ \ \ \ \  \langle b,a\cdot f \rangle=\langle ba,f \rangle\\
\langle n\cdot f,a \rangle=\langle n,f\cdot a \rangle   &  \ \ \ \ \ \ \ \langle a,f\cdot n \rangle=\langle a\cdot f,n \rangle\\
\langle m\square n,f \rangle=\langle m,n\cdot f \rangle & \ \ \ \ \ \ \ \langle f,m\lozenge n\rangle=\langle f\cdot m,n \rangle.
\end{array}$$
If $\square$ and  $\lozenge$ coincide on the whole of $A^{**}$ then  $A$ is called Arens regular.
 For the brevity  of notation we use the same symbol $``\cdot "$
for the various module operations linking $A$, such as $A^*$, $A^{**}$ and also as well for the $n^{\rm th}$ dual $A^{(n)}, (n\in\mathbb{N}).$ The main properties of these products and various $A-$module operations are detailed in \cite{A}; see also \cite{MV}.

A derivation from a Banach algebra $A$ to a
  Banach $A$-module $X$ is a bounded linear mapping $D:A\rightarrow X$
  such that  $D(ab)=D(a)b+aD(b)\ \ \ (a,b\in A)$. For each $x\in
  X$ the  mapping $\delta_x:a\rightarrow ax-xa$,  $(a \in A)$  is a
  bounded derivation, called an inner derivation. The concept of $n$-weak amenability was introduced  and intensively studied  by Dales {\it et al.} \cite{DGG}. A Banach algebra ${\mathcal
A}$ is said to be $n$-weakly amenable ($n\in\mathbb{N}$) if every
 derivation from ${\mathcal A}$ into ${\mathcal A}^{(n)}$ is
inner. Trivially, $1$-weak amenability is nothing else than weak
amenability.
  A derivation $D: A\rightarrow A^*$ is called   cyclic if
  $\langle D(a),b \rangle+\langle D(b),a \rangle=0 $  $(a,b\in A)$. If every bounded cyclic  derivation from $A$ to $A^*$  is inner then $A$ is called cyclicly
  amenable which was studied by Gr{\o}nb{ae}k \cite{G}.
 Throughout the paper we usually identify an element of a space with its canonical image in its second dual.\\

Now we come to $_\varphi A.$
A direct verification reveals that for  $a\in A, f\in (_\varphi A)^*$  and    $m,n\in  (_\varphi A)^{**},$
$$\begin{array}{ll}
 f\cdot a = \langle\varphi, a\rangle f         & \ \ \ \ \ \ \  a\cdot f =\langle f, a\rangle\varphi\\
 n\cdot f = \langle n, f\rangle\varphi   &  \ \ \ \ \ \ \ f\cdot n =\langle n,\varphi\rangle f\\
m\square n = \langle m, \varphi\rangle n & \ \ \ \ \ \ \ m\lozenge n = \langle m, \varphi\rangle n.
\end{array}$$
 The same calculation gives  the $_\varphi A-$module operations of $(_\varphi A)^{(2n-1)}$ and $(_\varphi A)^{(2n)}$  as follows,
$$\begin{array}{ll}
 F\cdot a = \langle\varphi, a\rangle F        & \ \ \ \ \ \ \  a\cdot F =\langle F, a\rangle\varphi  \ \ \ \ \ (F\in (_\varphi A)^{(2n-1)})\\
 G\cdot a = \langle G, \varphi\rangle  a &  \ \ \ \ \ \ \ a\cdot G =\langle\varphi, a\rangle G  \ \ \ \ \ \ \ (G\in
(_\varphi A)^{(2n)}).\\
\end{array}$$

${}$\\

We commence with the next straightforward result, most parts of which are  based on the latter observations on the various duals of $_\varphi A.$
\begin{proposition}
$(i)$ $_\varphi A$ is Arens regular and $({_\varphi A})^{**}=\ _\varphi (A^{**}).$  Furthermore,  for each $n\in\mathbb{N},$ $({_\varphi A})^{(2n)}$ is Arens regular.

$(ii)$  $(_\varphi A)^{**}\cdot\ _\varphi A=\ _\varphi A$ and $_\varphi A\cdot(_\varphi A)^{**}=(_\varphi A)^{**};$ in particular, $_\varphi A$ is a left  ideal of $(_\varphi A)^{**}.$

$(iii)$   $(_\varphi A)^*\cdot\ _\varphi A=(_\varphi A)^*$ and  $_\varphi A\cdot
(_\varphi A)^*=\mathbb{C}\varphi.$

\end{proposition}
As $_\varphi A$ has no approximate identity, in general, it is not amenable.
The next result  investigates  $n-$weak amenability of $_\varphi A$.
\begin{theorem}\label{k} For each $n\in\mathbb{N},  $ $_\varphi A$ is  $(2n-1)$-weakly
amenable, while in the case where $\dim(\ker\varphi)\geq 2,$  $_\varphi A$  is not $(2n)-$weakly amenable for any $n\in\mathbb{N}$.
\end{theorem}
\begin{proof} Let $D:\ _\varphi A \rightarrow (_\varphi A)^{(2n-1 )}$ be a
derivation and  let $a,b\in\ _\varphi A.$ Then
$$\langle\varphi
, a\rangle D(b)=D(ab)=D(a)b+aD(b)=\langle\varphi, b\rangle D(a)+\langle D(b), a\rangle\varphi.$$ It follows
that  $\langle\varphi
, a\rangle\langle D(b), a\rangle=\langle\varphi
, b\rangle\langle D(a), a\rangle+ \langle\varphi
, a\rangle\langle D(b), a\rangle,$
  from which we have $\langle D(a),a \rangle= 0$, or equivalently,
\ $\langle D(a+b), a+b\rangle=0.$ Therefore $\langle D(a), b\rangle=-\langle D(b), a\rangle$. Now with $e$ as a left identity for $_\varphi A$  we have
$$D(b)=D(eb)=\langle\varphi, b\rangle D(e)+\langle D(b), e\rangle\varphi=\langle\varphi, b\rangle D(e)-\langle D(e), b \rangle\varphi=
\delta_{-D(e)}(b).$$ Therefore $D$ is inner, as required.

To prove that $_\varphi A$ is not $(2n)-$weakly amenable for any $n\in\mathbb{N}$,  it is enough  to show that
$_\varphi A$  is not 2-weakly amenable, {\cite[Proposition 1.2]{DGG}}. To this end let $f\in
(_\varphi A)^*$ such that $f$  and   $\varphi$ are linearly
independent. It follows that
$\langle f, a_0\rangle=\langle\varphi, b_0\rangle=0$ and $\langle f, b_0\rangle=\langle\varphi, a_0\rangle=1,$ for some $a_0, b_0\in\ _\varphi A$.  Define $D:\ _\varphi A\rightarrow
(_\varphi A)^{**}$ by $D(a)=\langle f-\varphi, a\rangle b_0$, then  $D$ is a
derivation. If there exists $m\in (_\varphi A)^{**}$ \ with  \
$D(a)=am-ma$\ $(a\in\ _\varphi A)$ then  $ b_0=-\langle m, \varphi\rangle b_0$ which follows that
$1=-\langle m, \varphi\rangle$. Now if $a\in \ker\varphi$ then
$\langle f, a\rangle b_0=-\langle m, \varphi\rangle a=a.$   It follows that $\dim(\ker\varphi)=1$
that is a contradiction.
\end{proof}
As an immediate consequence of Theorem \ref{k} we give the next result of Zhang.
\begin{corollary} [{\cite[Page 507]{Z}}] For each $n\in\mathbb{N},$ $_{\varphi_{s_0}}\l^1(S)$ is $(2n-1)-$weakly amenable, while it is not $(2n)-$weakly amenable for any $n\in\mathbb{N}$.
 \end{corollary}

\begin{proposition}
A bounded linear map $D:\ _\varphi A \rightarrow (_\varphi A)^{(2n)}, (n\in\mathbb{N}),$ is a
derivation if and only if $D(_\varphi A)\subseteq \ker\varphi$.
\end{proposition}
\begin{proof}
A direct verification shows that   $D:\ _\varphi A\rightarrow (_\varphi A)^{(2n)}$ is a derivation if and only if
$$\langle\varphi
, a\rangle D(b)=D(ab)=D(a)b+aD(b)=\langle D(a), \varphi\rangle b+\langle\varphi, a\rangle D(b)\ \ \ \ \ (a,b\in\ _\varphi A).$$
And this is equivalent to $\langle D(a),\varphi \rangle= 0, \ (a\in\ _\varphi A);$ that is $D(_\varphi A)\subseteq
\ker\varphi$. Note that here $\varphi$ is assumed to be an element of $(_\varphi A)^{(2n+1)}.$
\end{proof}
The next results demonstrates that in contrast to Theorem \ref{k},  $_\varphi A$ is $(2n)-$weakly amenable  in the case where $\dim(\ker\varphi)< 2.$
\begin{proposition}\label{p} If $\dim(\ker\varphi)< 2 $ then $_\varphi A$  is $(2n)-$weakly amenable for each $n\in\mathbb{N}$.
\end{proposition}
\begin{proof} The only reasonable  case  that we need to verify  is $\dim(\ker\varphi)=1.$ In this case we have $\dim (A)=2$.  Therefore one  may assume   that $A$ is generated by two elements $e, a\in A$ such that $\langle\varphi, e\rangle=1$ and $\langle\varphi, a\rangle=0.$ Let $f\in
(_\varphi A)^*$ satisfy  $\langle f, e\rangle=0$ and $\langle f, a\rangle=1.$ Then $f$  and   $\varphi$ are  linearly independent and generate $A^*$; indeed, every  non-trivial element $g\in A^*$ has the form $g=\langle g, e\rangle\varphi+\langle g, a\rangle f$.  Let $D:\ _\varphi A\rightarrow (_\varphi A)^{(2n)}$ be a derivation then as Proposition \ref{p} demonstrates $D(_\varphi A)\subseteq \ker\varphi.$ Therefore $D(x)=\langle g, x\rangle a, (x\in\ _\varphi A)$, for some $g\in (_\varphi A)^*.$ As for every $x\in\ _\varphi A,$  $x=\langle\varphi, x\rangle e+\langle f, x\rangle a$, a direct calculation reveals that $D=\delta_{(\langle g, e\rangle a-\langle g, a\rangle e)}$; as required.
\end{proof}
\vskip 0.4 true cm
\begin{remark} $(i)$ If we go through the proof of Theorem \ref{k} we see that the range of the derivation   $D(a)=\langle f-\varphi, a\rangle b_0$ lies in $_\varphi A$, and the same argument may apply to show that it  is not inner as a derivation from $_\varphi A$ to
$_\varphi A.$ This shows that $_\varphi A$ is not  $0-$weakly amenable, i.e. $H^1(_\varphi A,\ _\varphi A)\neq 0;$ see a remark just after {\cite[Proposition 1.2]{DGG}}. The same situation has occurred   in the proof of  Proposition \ref{p}.

$(ii)$ As $_\varphi A$ has a left identity and it is a left ideal in $(_\varphi A)^{**}$,   it is worthwhile mentioning that,  to prove the $(2n-1)-$weak amenability of $_\varphi A$ it suffices to show that $_\varphi A$ is weakly amenable; {\cite[Theorem 3]{Z}}, and this has already  done by  Dales {\em et al.} {\cite[Page 713]{DPV}}.
\end{remark}

\vskip 0.4 true cm

We have seen in the first part of the proof of Theorem \ref{k} that if  $D:\ _\varphi A \rightarrow (_\varphi A)^{(2n-1 )}$ is a derivation then $\langle D(a), b\rangle+\langle D(b), a\rangle=0, \ (a,b\in\ _\varphi A);$ and the latter is known as a cyclic derivation for the case $n=1.$ Therefore as a consequence of Theorem \ref{k} we get:
\begin{corollary} A bounded linear mapping $D:\ _\varphi A\rightarrow (_\varphi A)^*$ is a derivation if and only if it is a cyclic derivation. In particular, $_\varphi A$ is cyclicly amenable.
\end{corollary}
\vskip 0.4 true cm

 We conclude with  the following list consist of some   miscellaneous properties of $_\varphi A$ which can be verified straightforwardly.

$\bullet$ If $\varphi=\lambda \psi$ for
some   $\lambda\in\mathbb{C}$ then trivially  $_\varphi A$ and  $_\psi A$ are isomorphic; indeed, the mapping $a\rightarrow\lambda a$ defines an isomorphism. However, the converse is not valid, in general. For instance, let $A$ be generated by two elements $a, b.$ Choose $\varphi, \psi\in A^*$ such that $\langle\varphi, a\rangle=\langle\psi, b\rangle=0$ and $\langle\varphi, b\rangle=\langle\psi, a\rangle=1,$ then $_\varphi A$ and  $_\psi A$ are isomorphic (indeed, $\alpha a+\beta b\rightarrow \alpha b+\beta a$
defines an isomorphism), however $\varphi$ and $\psi$ are linearly independent.

$\bullet$ It can be readily verified that $\{0\}\cup\{a\in\ _\varphi A,
  \varphi(a)=1 \}$ is the set of all idempotents of $_\varphi A$. Moreover, this is actually the set of all minimal idempotents of $_\varphi A$.

  $\bullet$ It is obvious that  every subspace  of $_\varphi A$ is a left ideal, while a subspace $I$ is a right ideal if and only if either $I=\ _\varphi A$ or $I \subseteq \ker\varphi.$ In particular,
$\ker\varphi$ is the unique maximal ideal in $_\varphi A$.

$\bullet$ A subspace $I$ of $_\varphi A$ is a modular left  ideal if and only if either $I= _\varphi A$ or $I= ker\varphi.$ In particular,
$\ker\varphi$ is the unique primitive ideal in $_\varphi A$ and this implies that  $rad(_\varphi A)=\ker\varphi$ and so  $_\varphi A$ is not semi-simple. Furthermore, for every non-zero proper closed ideal $I$, $rad(I)=I\cap rad(_\varphi A)=I\cap\ker\varphi=I.$

$\bullet$ A direct verification reveals that $LM(_\varphi A)=\mathbb{C}I$ and $RM(_\varphi A)=B(_\varphi A)$, where $LM$ and $RM$  are stand for the  left and right multipliers, respectively.

\end{document}